\documentclass[a4paper,12pt]{amsart}

\usepackage[margin=1.5in]{geometry}

\usepackage{todonotes}

\usepackage[utf8]{inputenc}
\usepackage[T1]{fontenc}
\usepackage[english]{babel}
\usepackage{csquotes}
\usepackage{amssymb}
\usepackage{amsmath,amsthm}
\usepackage{datetime}

\usepackage{draftwatermark}
\SetWatermarkText{}
\SetWatermarkScale{0.6}

\usepackage[shortlabels]{enumitem}
\usepackage[pdfauthor={David Schrittesser},
            pdftitle={Maximal discrete sets},
            pdfproducer={Latex with hyperref}]{hyperref}

\theoremstyle{plain}
\newtheorem{theorem}{Theorem}[section]
\newtheorem{lemma}[theorem]{Lemma}

\newtheorem{corollary}[theorem]{Corollary}

\newtheorem{examples}[theorem]{Examples}
\newtheorem{conjecture}[theorem]{Conjecture}

\newtheorem{question}[theorem]{Question}

\newtheorem*{claim*}{Claim}
\newtheorem*{theorem*}{Theorem}
\newtheorem*{lemma*}{Lemma}
\newtheorem*{proposition*}{Proposition}
\newtheorem*{corollary*}{Corollary}

\theoremstyle{definition}
\newtheorem{definition}[theorem]{Definition}

\newtheorem{fact}[theorem]{Fact}

\newtheorem{remark}[theorem]{Remark}

\newtheorem{convention}[theorem]{Convention}

\newtheorem*{definition*}{Definition}
\newtheorem*{example*}{Example}
\newtheorem*{fact*}{Fact}

\theoremstyle{remark}

{\end{minipage}\end{equation}}

\newenvironment{eqpar*}{\begin{equation*}\begin{minipage}{0.8\columnwidth}}%
{\end{minipage}\end{equation*}}

\DeclareMathOperator{\fix}{fix}

\newcommand{\Q}{\mathbb{Q}}
\newcommand{\bF}{\mathbb{F}}

\DeclareMathOperator{\powerset}{\mathcal{P}}

\DeclareMathOperator{\rat}{\mathbb{Q}}

\providecommand{\Int}{\mathbb{Z}}
\providecommand{\reals}{\mathbb{R}}

\DeclareMathOperator{\dom}{dom}
\DeclareMathOperator{\lh}{lh}

\DeclareMathOperator{\ran}{ran}

\DeclareMathOperator{\sgn}{sgn}

\providecommand{\res}{\mathbin{\upharpoonright} }
\providecommand{\conc}{ \mathbin{{}^\frown}}

\providecommand{\la}{ \langle}
\providecommand{\ra}{ \rangle}

\newcommand{\fin}{\ensuremath{\textsf{\textup{Fin}}}}

\providecommand{\CH}{\ensuremath{\textup{\textsf{CH}}}}
\providecommand{\GCH}{\ensuremath{\textup{\textsf{GCH}}}}
\providecommand{\ZFC}{\textup{\textsf{ZFC}}}
\providecommand{\ZF}{\textup{\textsf{ZF}}}
\providecommand{\HOD}{\mathbf{HOD}}

\DeclareMathOperator{\eL}{\mathbf{L}}
\DeclareMathOperator{\Ve}{\mathbf{V}}

\providecommand{\DC}{\textsf{\textup{DC}}}

\providecommand{\setdef}{\;|\;}

\DeclareMathOperator{\On}{On}

\providecommand{\Coll}{\textup{Coll}}

\author{David Schrittesser}
\title[Maximal discrete sets]{Maximal discrete sets}

\address{David Schrittesser, Kurt G\"odel Research Center, University of Vienna, Augasse 2--6, UZA 1, Building 2, 1090 Vienna, Austria}
\email{david@logic.univie.ac.at}

\date{\today\  (\xxivtime)}


\begin{document}

\begin{abstract}
We survey results regarding the definability and size of maximal discrete sets in analytic hypergraphs.
Our main examples include maximal almost disjoint (or mad) families, $\mathcal I$-mad families, maximal eventually different families, and maximal cofinitary groups. We discuss the non-increasing sequence of cardinal characteristics $\mathfrak a_\xi$, for $\xi<\omega_1$ as well as the notions of spectra of characteristics and optimal projective witnesses. We give an account of Zhang's forcing to add generic cofinitary permutations, and of a version of this forcing with built-in coding. 
\end{abstract}

\maketitle
%\tableofcontents

As probably every mathematician and certainly every logician knows, with the Axiom of Choice one can construct objects which have counter-intuitive, often unrealistic, even fantastic properties---such as the Banach-Tarski decomposition. %see, eg. \cite{}.
These mathematical objects have been called paradoxical and even ``monsters''---a wonderful and fitting metaphor! 
%I will discuss some recent and a few less recent results about these monsters in the present article, which is to be viewed mostly as a survey, but with a some useful technical results which are not published (or perhaps, publishable) anywhere else.

In this article I will discuss a number of such objects, whether they be deemed monstrous or not, arising as \emph{maximal discrete sets for hypergraphs}.
Largely, this article is to be viewed as a survey or as an attempt to lure new contributors into this area. I will also announce one or two theorems, the proof of which is left for a future paper, as well as make some conjectures.
I also give some partial results which are not published anywhere else.

\medskip

%The paper is organized as follows.

\section{A menagerie of beautiful monsters}

Recall that a \emph{hypergraph} is a a structure $G=(X,H)$ where $X$ is a set and $H \subseteq \powerset(X)$; elements of $X$ are called the \emph{vertices} of $G$, and elements of $H$ are called its \emph{hyperedges}. 
The cardinality of a hyperedge is called its \emph{arity}. 
To say that $G$ is \emph{$\kappa$-uniform}, where $\kappa$ is a cardinal, means that $\kappa$ is the arity of each of its hyperedges.
A hypergraph is called \emph{simple} if $H$ does not contain singletons, i.e., no loops (hyperedges connecting a vertex only to itself). 
Thus, we can define a \emph{simple graph} $G=(X,E)$ to be a $2$-uniform hypergraph; an \emph{edge} is a hyperedge of arity $2$.
We similarly write $<\kappa$-uniform and $\leq \kappa$-uniform to mean the obvious things.

We will be interested mostly in simple graphs, and simple hypergraphs where the arity of each hyperedge is some finite number, i.e., $<\omega$-uniform simple hypergraphs. 

\begin{convention}
In this note, \emph{hypergraph} means \emph{simple $<\omega$-uniform hypergraph} and \emph{graph} means \emph{simple graph}. 
\end{convention}

\medskip

A set $D\subseteq X$ is called $G$-discrete if
$\powerset(D)\cap H = \emptyset$, that is, if no subset of $D$ forms a hyperedge.
In the case of graphs, this is clearly the same as saying that no two (or no two distinct\footnote{This is equivalent since we are only considering simple graphs.}) elements of vertices from $D$ form an edge.
A \emph{maximal $G$-discrete set} is a discrete set which is maximal among discrete sets with respect to $\subseteq$. We will sometimes just say \emph{(maximal) discrete set} if the hypergraph in question is clear from the context.
%Maximal discrete sets are also called \emph{kernels}.

\medskip

Now let us start with some examples of graphs.

\begin{examples}~\label{e.graphs}
\begin{enumerate}
\item $G=({}^\omega\omega,E_0)$ recalling that $E_0$ is the equivalence relation given by 
\[
x \mathbin{E_0} y \iff \lvert\{n\in\omega\setdef x(n)\neq y(n)\}\rvert<\omega. 
\]
Similarly, one can consider $G=({}^\omega 2, E_0)$, where we sloppily also write $E_0$ for the restriction of this relation to ${}^\omega 2$.
\item $G=(\reals, E)$ where $E$ is given by the equivalence relation $r \mathbin{E} s \iff r-s \in \rat$. A maximal discrete set is called a \emph{Vitali set}.
\item For this and the next two examples let $X=[\omega]^\omega$. Consider $G=(X,E)$ where
$A\mathbin{E} B \iff \lvert A \cap B\rvert =\omega$, i.e., if $A$ and $B$ are \emph{not} almost disjoint.  A maximal discrete set is known as a \emph{mad family}.\footnote{We do \emph{not} define ``mad family'' to mean an \emph{infinite} maximal $G$-discrete set.}
%\item Now let $\mathcal I$ be an ideal on $\omega$ and consider  $A\mathbin{E} B \iff A \cap B\notin \mathcal I$. A maximal discrete set is called an $\mathcal I$-mad family.

\item $(G,E)$ where $X={}^\omega\omega$ and $f E g$ if the graphs of $f$ and $g$ are not almost disjoint. If $\neg(f \mathbin{E} g)$ we also say $f$ and $g$ are \emph{eventually different}; a maximal discrete set is called a \emph{maximal eventually different family}.
\end{enumerate}
\end{examples}

We continue by giving some examples of hypergraphs.

\begin{examples}~\label{e.hypergraphs}
\begin{enumerate}
\item Let $X=\reals$ and $\{r_0, \hdots, r_n\} \in H$ if $r_0, \hdots, r_n$ are linearly dependent in $\reals$ seen as a vector space over $\rat$. A maximal discrete set is called a \emph{Hamel basis} of $\reals$ over $\rat$.
\item Let again $X=[\omega]^\omega$. For $s^+, s^- \in {}^{<\omega}X$ write
\[
B(s^+,s^-) = \bigcup \ran(s^+) \setminus \bigcup \ran(s^-)
\]
and let $\{A_0, \hdots, A_n\} \in H$ if 
for all $s^+, s^- \in {}^{<\omega}\{A_0, \hdots, A_n\}$, the set $B(s^+,s^-)$ is finite.
An $E$-discrete family is called an \emph{independent family (of sets)} and a maximal discrete set is called a \emph{maximal independent family (of sets)}. A similar graph has been consisered on ${}^\omega\omega$ (see, e.g., \cite{miller}).
\item Let $X=S_\infty$ (the group of permutations of $\omega$) and let 
$\{g_0, \hdots, g_n\} \in H$ if $\la g_0, \hdots, g_n\ra$, the subgroup of $S_\infty$ generated by $g_0, \hdots, g_n$, is not \emph{cofinitary}, i.e., it contains a permutation  with infinitely many fixed points other than the idendity.
A maximal discrete set is a \emph{maximal cofinitary group}. We write $G_{\mathrm{mcg}}$ for this hypergraph.
\end{enumerate}
\end{examples}

In all cases discussed in this article, the set of vertices $X$ is an effective Polish metric space, so we can use the projective hierarchy in its boldface and lightface version to measure definitional complexity. 
For all of the spaces $X$ we discuss, we can regard $\powerset([X]^{<\omega})$ as an effective Polish metric space
by indentifying it with the set $\{ \bar x \in {}^{<\omega} X \setdef \bar x_0 < \hdots < \bar x_{\lh(\bar x)-1}\}$, with the usual Polish metric structure, where $<$ is the natural strict linear order on the space $X$ in question.
Of course when we say that a hypergraph $G=(X,H)$ is $\Sigma^0_\xi$, Borel, analytic or $\mathbf{\Sigma}^1_1$, (lightface) $\Sigma^1_1$ etc., we mean that $X$ is an (effective) Polish (metric) space and the set of hyperedges $H$ is $\Sigma^0_\xi$, Borel, etc.\ as a subset of
$[X]^{<\omega}$.\footnote{We lack a good example of a $\leq\omega$-uniform hypergraph.} 
%Unfortunately, the author cannot think of an interseting example of such a hypergraph, or of a maximal discrete set in such a hypergraph! If the reader should know of one, they are encouraged to contact the author.

\medskip

Obviously, all of the above graphs and hypergraphs are Borel. %with the sole exception of the graph $?$ giving rise to $\mathcal I$-mad families as maximal discrete sets. In this case the complexity of course depends on the complexity of $\mathcal I$.
We will later write down an interesting family of graphs whose complexity lies cofinally in the Borel hierarchy (cf.\ also \cite{karen})

\medskip

Of course, you might say, one can define a number of combinatorial objects in the language of hypergraphs, simply because this language is so versatile. The benefit, you might continue, of doing so is elusive since the objects in question are in fact very different.
To this I answer that to the contrary, indeed theorems can be proved at precisely this level of generality, such as the following result due to Vidny\'ansky.

\medskip

First, let us introduce some terminology:
\begin{definition}
Let $G=(X,H)$ be a hypergraph, $x\in X$ and $C\subseteq X$.
We say \emph{$x$ is $G$-independent from $C$} if $C\cup\{x\}$ is discrete.
We say \emph{$x$ is $G$-caught by $C$} to mean that on the contrary, $C\cup\{x\}$ is \emph{not} discrete.
As before we shorten these to \emph{independent} and \emph{caught} when $G$ is clear from the context.
\end{definition}

\begin{theorem}[\cite{zoltan}]\label{t.zoltan}
Suppose $\Ve=\eL$ and we are given  an analytic hypergraph $G=(X,H)$ on a Polish space which satisfies the following property:

For any $z \in {}^\omega 2$, any countable discrete set $C \subseteq X$ and any $x \in X$ which is independent from $C$,
there are $y_0,\hdots,y_n \in X$ such that 
\begin{itemize}
\item $C \cup\{y_0, \hdots, y_n\}$ is discrete but
\item $x$ is caught by $C \cup\{y_0, \hdots, y_n\}$  and 
\item for each $i<n+1$ it holds that $z$ is computable  in $x_i$ (or just hyperarithmetic  in $x_i$).
\end{itemize}
Then there is a $\mathbf{\Pi}^1_1$ maximal $H$-discrete set. 
%If $X$ is an effective Polish metric space and $G$ is $\Sigma^1_1(r)$ there is a $\Pi^1_1$ maximal $H$-discrete set. 
\end{theorem}

The above is only a \emph{special case} of the theorem proved in \cite{zoltan}: Restricting our attention to maximal discrete sets makes the statement much \emph{less} abstract. 
Earlier versions of this construction, carried out for specific hypergraphs can be found in Arnie Miller's famous article \cite{miller}.
%This can be improved as follows:

%\begin{theorem} %The BPFA theorem? Do I do that?
%Suppose $H$ is an analytic hypergraph such that the following holds:
%For any countable set $C \subseteq

%Then in $\eL$, there is a $\Pi^1_1$ maximal $H$-discrete set.
%\end{theorem}

\medskip

Moreover once we are at this level of generality, i.e., at bird-eyes' view, interesting questions can be asked:

\begin{question}\label{q.measurable}
Is there a hypergraph $G$ such that the statement ``there does not exist a maximal  $G$-discrete set'' has consistency strength 
at least a measurable cardinal? What about even stronger large cardinal notions?
\end{question}

This question is motivated by the following result of Ho\-ro\-witz and Shelah (see \cite{horowitz-shelah-graphs}):
\begin{theorem}
There is a Borel graph $G$ such that the $\ZF + \DC$ $+$ ``there does not exist a maximal $G$-discrete set'' implies that $\omega_1$ is inaccessible in $\eL$.
\end{theorem}

Of course it is tempting to find an upper bound for the type of large cardinals appearing in Question~\ref{q.measurable}. 
Such an upper bound could be given by constructing a model in which some large class of hypergraphs do not have maximal discrete sets. But one has to be careful to exclude graphs which have maximal discrete sets provably in $\ZF$.

There are of course trivial such cases. For instance, the graph giving rise to mad families has finite maximal discrete sets.
A less trivial case is given by maximal eventually different families and maximal cofinitary groups, which were shown by Horowitz and Shelah to exist just working in $\ZF$ (see below, Theorems~\ref{borel.mcg.medf} and~\ref{my.mcg.medf}).

The following formulation, as far as I can tell, takes these obstacles into account and thus (hopefully) renders the question non-trivial:
\begin{question}\label{q.model.no.discrete}
Is there a model of $\ZFC$ where no projective hypergraph has a maximal $H$-discrete set which is $\Pi^1_1$ hard?
Is there a model of $\ZF$ where no hypergraph has a maximal $H$-discrete set which is $\Pi^1_1$ hard?
\end{question}
By the previous theorem, to construct such a model one must assume at least that $\ZFC$ is consistent with the existence of an inaccessible. 
Any answer to Question~\ref{q.model.no.discrete} would give an upper bound for what kind of large cardinal property can be substituted for measurability in Question~\ref{q.measurable}.

\section{Regularity and definability of maximal discrete sets}

It is nevertheless true that some of the examples listed above differ significantly from the others.
Namely, in all but two cases one can show that some (perhaps weak) form of the Axiom of Choice is necessary to produce maximal discrete sets for the above graphs: 
For all these hypergraphs $G$, one can construct a model of $\ZF$ without any maximal $G$-discrete sets by forcing. 
In fact usually this will hold in Solovay's model (that is in $\HOD({}^\omega\On)$ of the the generic extension by $\Coll(\omega,<\kappa)$ where $\kappa$ is inaccessible in the ground model).
In some cases we know that this is not the optimal proof in terms of consistency strength. 
For example, a model where every set has the Baire property can be obtained just from the assumption that $\ZF$ is consistent, and in this model there will be no maximal independent family of sets, no Hamel basis, no Vitali set, and no transversal for $E_0$.
To give a more recent example, Horowitz and Shelah \cite{horowitz-shelah-inacc} showed the following:

\begin{theorem}
Every model of $\ZF+\DC$ has a forcing extension in which $\DC$ holds and there is exists no infinite mad family which is $\mathbf{OD}({}^\omega\On)$. In particular, if $\ZF$ is consistent, so is $\ZF+\DC$ $+$ there is no infinite mad family.
\end{theorem}

For many hypergraphs $G$ it can moreover be shown that a $G$-maximal discrete set cannot be analytic (for some, as we have seen the proof goes through standard regularity properties such as measurability or the Baire property).

In contrast, as was already mentioned, the following two results were shown in 2016 by Horowitz and Shelah \cite{medf-borel,mcg-borel}:
\begin{theorem}\label{borel.mcg.medf}
In $\ZF$ one can construct a Borel maximal eventually different family as well as a Borel maximal cofinitary group.
\end{theorem}
It is easy to see that for any analytic hypergraph, the existence of a Borel maximal discrete set and of an analytic such set are equivalent.

\medskip

In fact much simpler maximal discrete sets can be constructed for these two particular hypergraphs:
\begin{theorem}\label{my.mcg.medf}
In $\ZF$ one can construct
\begin{itemize}
\item A closed maximal eventually different family \cite{medf-closed,medf-closed-v2,medf-compact}. %medf-closed-v2,
\item A maximal cofinitary group $\mathcal G$ with a set of generators which is closed in $S_\infty$; thus $\mathcal G$ is $\Sigma^0_2$. 
\end{itemize}
\end{theorem}

\medskip

So far, we have investigated whether the Axiom of choice is needed to prove the existence of a $G$-maximal discrete set for a given hypergraph $G$.
At least three other types of interesting questions can be asked:
\begin{question}
Given a hypergraph $G$, does the existence of a maximal $G$-discrete set imply a \emph{fragment of the Axiom of Choice} (such as $\textsf{\textup{AC}}_\omega$) in $\ZF$?
\end{question}
Negative answers for such a question are, of course, independence results.
Examples for research addressing this type of question can be found in a series of research articles by Beriashvili, Castiblanco, Kanovei, Schindler, Wu, and Yu \cite{schindler1, schindler2,schindler3}. 
\begin{question}
Given a hypergraph $G$, does the existence of some (or some infinite, or even uncountable) maximal $G$-discrete set imply the existence of an irregular set, such as a non-measurable set, a set without the Baire property etc.?
\end{question}
As we have seen for many hypergraphs $G$ the existence of a maximal $G$-discrete set implies there are non-measurable sets and sets without the Baire property. 
It is shown in \cite{pnas,higher-ramsey} that if there is an infinite $\fin^\alpha$-mad familiy, there is a set which behaves irregularly in a Ramsey theoretic sense.
On the other hand, Horowitz and Shelah \cite{horowitz-shelah-graphs} show that the non-existence of a mad family is consistent with a broad assumption of universal regularity, namely measurability with respect to $\omega^\omega$-bounding forcing notions.
\begin{question}
Given two hypergraphs $G$, $G'$ does the existence of a maximal $G$-discrete set imply, in $\ZF$, the existence of a maximal $G'$-discrete set?
\end{question}
For instance, one can show in $\ZF$ that if there is a Hamel basis, then there is a Vitali set.

%
%
% I MAD
%

\medskip

\subsection{Mad families and their higher dimensional relatives}

We shall now generalize the definition of mad family and construct a class of graphs with complexity cofinal in the Borel hierarchy.
The generalization of mad families is achieved by allowing an arbitrary ideal $\mathcal I$ on $\omega$ instead of
the ideal $\fin$ of finite sets. 
\begin{definition}
Let $\mathcal I$ be an ideal on a countable set $S$. An \emph{$\mathcal I$-almost disjoint} set is a discrete set in the hypergraph $G=([S]^\omega,E)$ where $A \mathbin{E} B \iff A\cap B\notin \mathcal I$. An \emph{$\mathcal I$-mad (maximal almost disjoint) family} is a maximal discrete set in $G$.
\end{definition}
The following questions are open.\footnote{A previous version of \cite{karen} claimed to have answered the first question, but alas, a mistake was found (by the authors themselves).}
An $F_\sigma$ ideal is simply an ideal which is $F_\sigma$ as a subset of ${}^\omega 2$. 
For the other notions used in the question, see, e.g., \cite{farah}.
\begin{question}
Does there exist an analytic (equivalently, Borel) $\mathcal I$-mad family, where $\mathcal I$ is some $F_\sigma$ ideal?
What about if (instead, or in addition) $\mathcal I$ is assumed to be an ideal of the form $\operatorname{Exh}(\phi)$, for some lower continuous submeasure $\phi$ on $\omega$?
\end{question}
Ideals coming from lower continuous submeasures lie very low in the Borel hierarchy. 
To obtain more complex ideals, we can use Fubini sums.
\begin{definition}~
\begin{enumerate}
\item Suppose $\mathcal I$ is an ideal on $\omega$ and for each $n$, $\mathcal I_n$ is an ideal on a countable set $S_n$ (which we call its \emph{underlying set}).
Then $\bigoplus_{\mathcal I} \mathcal I_n$, the \emph{Fubini sum of $\la \mathcal I_n \setdef n\in \omega\ra$ over $\mathcal I$} is the ideal on $\bigsqcup_n S_n = \bigcup\{n\}\times S_n$ given by
\[
X \in \bigoplus_{\mathcal I} \mathcal I_n \iff \{n\in\omega \setdef X(n)\notin \mathcal I_n\}\in\mathcal I
\]
where we write $X(n)$ for $\{s\in S_n\setdef (n,s)\in X\}$, the ``$n$th vertical'' of $X$.
\item If $\mathcal I_n=\mathcal J$ for each $n\in\omega$ we write
\[
\mathcal I \otimes \mathcal J=\bigoplus_{\mathcal I} \mathcal I_n.
%$
%\mathcal I \otimes \mathcal J$ for $\bigoplus_{\mathcal I} \mathcal I_n.
%$
\] 
Moreover, for $n\in\omega$ we write
$\mathcal I^n$ for the $n$-fold $\otimes$-product of $\mathcal I$ with itself:
\[
\mathcal I^n = \underbrace{\mathcal I \otimes \hdots \otimes \mathcal I}_\text{$n$ times}
\]
\item Write $\mathfrak F$ for the smallest collection of ideals such that $\fin\in\mathfrak F$ and if $\mathcal I_n \in \mathfrak F$ for each $n\in \omega$ then $\bigoplus_\fin \mathcal{I}_n \in \mathfrak F$. 

\end{enumerate}
\end{definition}
The sequence of ideals $\fin^n$ for $n\in\omega$ is readily seen to be increasing in Borel complexity. 
Obviously the ideals in $\mathfrak F$ extend this sequence into the transfinite.
We shall later explore the structure of $\mathfrak F$ in a more detail.
At present, let us point out that the methods in \cite{karen} show the following:
\begin{theorem}
For no $\mathcal I \in \mathfrak F$ does there exist an infinite analytic $\mathcal I$-mad family.
\end{theorem}
In fact the methods of \cite{higher-ramsey} show the following stronger theorem:
\begin{theorem}
Assume $\ZF+\DC$ and suppose $\Gamma$ is a pointclass closed under taking continuous preimages and intersections with Borel sets and such that every total binary relation on $[\omega]^\omega$ in $\Gamma$ can be uniformized on a non-empty Ellentuck open set by a continuous function. 
Let $\mathcal I \in \mathfrak F$. Then there is no infinite $\mathcal I$-mad family in $\Gamma$.
\end{theorem}
As a corollary to this theorem, there is no infinite $\mathcal I$-mad family in Solovay's model for $\mathcal I \in \mathfrak F$.
Moreover, under the Axiom of Determinacy there is no such family in $\eL(\reals)$, and under the Axiom of Projective Determinacy, there is no projective (Dedekind-infinite) such family.

\section{Combinatorial cardinal characteristics and spectra}

Let a hypergraph $G=(X,H)$ be given.
\begin{question}
What are the possible sizes of $G$-maximal discrete sets? How are they related to other combinatorial cardinal characteristics of the continuum?
\end{question}
Let us make the following definition.
\begin{definition}
Write 
\begin{gather*}
\operatorname{Spec}(G) = \{\lvert D \rvert : \text{$D\subseteq X$, $D$ is a $G$-maximal discrete set}\},\\
\mathfrak a_G = \min\big(\operatorname{Spec}(G)\setminus\omega\big).
\end{gather*}
Conventionally, one writes $\mathfrak a$, $\mathfrak a_\mathrm{e}$ and $\mathfrak{a}_\mathrm{g}$ for the least sizes of a mad family, maximal eventually different family, and a maximal cofinitary group, respectively.
%\begin{align*}
%\mathfrak b &= \min\{ \lvert \mathcal F\rvert : \mathcal F \subseteq \omega^\omega, (\forall g \in \omega^\omega) (\exists f \in \mathcal F)\; f \mathbin{\not\leq^*} g\},\\
%\mathfrak d &= \min\{ \lvert \mathcal F\rvert : \mathcal F \subseteq \omega^\omega, (\forall g \in \omega^\omega) (\exists f \in \mathcal F)\; g \mathbin{\leq^*} f\}
We also write
\[
\mathfrak a_{\mathcal I} = \min\{\lvert \mathcal A\rvert : \mathcal A \subseteq \powerset(\omega), \text{ $\mathcal A$ is an infinite $\mathcal I$-mad family}\}.
\]
\end{definition}
\begin{remark}
Note that for some graphs, it may be more appropriate or more interesting to define $\mathfrak a_G$ as $\min\big(\operatorname{Spec}(G)\setminus\omega_1\big).$
\end{remark}

\subsection{Higher dimensional mad families as families on trees}

What can be said about $\mathfrak a_{\mathcal I}$ for $\mathcal I \in \mathfrak F$?
To address this question, we must first find a rough description of the ideals in $\mathfrak F$.
To this end, write $\mathfrak T_0$ for the collection of non-empty subtrees $T$ of ${}^{<\omega}\omega$ such that
\begin{itemize}
\item $T$ is well-founded
\item For each $t \in T$, the set $\{n\in\omega\setdef t\conc n\in T\}$ is either infinite or equal to $\emptyset$.
\end{itemize} 
Consider the equivalence relation $\sim$ on $\mathfrak T_0$ given by $T \sim T'$ if there is some bijection $f\colon T \to T'$ which preserves the length of sequences, and let $\mathfrak T=\{[T] \setdef T \in \mathfrak T_0\}$, where $[T]$ denotes the equivalence class with respect to $\sim$ of $T$.

Define for each $T\in\mathfrak T_0$ an ideal  $I(T)$ on $T$ by induction on the ordinal rank of well-founded trees: 
If $T$ is the unique tree in $\mathfrak T_0$ with rank $1$ let $\mathcal I(T)$ be the finite ideal on $T$, 
\[
\mathcal I(T) =  [T]^{<\omega}
\]
which is obviously the same as $\fin$ up to a bijection of the underlying sets.
For $T\in\mathfrak T_0$ of rank greater than $1$ let $\mathcal I(T)$ 
be given by
\[
X \in \mathcal I(T) \iff \{n\in \omega \setdef X_{\lfloor n \rfloor} \notin \mathcal I(T_{\lfloor n \rfloor})\}\in\fin
\]
where for $X \subseteq {}^{<\omega}\omega$ we use the notation
\[
X_{\lfloor n \rfloor} =\{t'\in {}^{<\omega}\omega\setdef \la n\ra \conc t' \in X\}.
\]
Note that obviously, up to a bijection of the underlying sets of ideals,
\[
\mathcal I(T) = \bigoplus_{\fin} \mathcal I(T_{\lfloor n \rfloor}).
\]
As we shall see, this gives us a useful presentation of the ideals in $\mathfrak F$:
\begin{fact}
$\mathfrak F = \{\mathcal I(T)\setdef T \in \mathfrak T_0\}$ and if $T\sim T'$ then $\mathcal I(T)=\mathcal I(T')$ where we consider ideals to be identical if they are the same up to a bijection of their underlying sets.
\end{fact}
Thus, with this slight abuse of notation we could write $\mathfrak T$ for $\mathfrak T_0 / \sim$ and define $\mathcal I([T])=\mathcal I(T)$ (as we identify ideals which agree up to renaming elements of the underlying space).

\medskip

But one should be aware here that $\mathcal I(T)$ and $\mathcal I(T')$ can be very similar also when $T \not\sim T'$.
To give a simple example, we have the following useful fact:
\begin{fact}\label{f.restrict}
Suppose $T, T' \in \mathfrak T_0$ and $T \setminus T' \in \mathcal I(T)$. Then 
\[
\mathcal I(T')=\{X \cap T'\setdef X\in T\}, 
\]
i.e., $\mathcal I(T')$ is $\mathcal I(T)$ restricted to $T'$, which since $T'$ agrees with $T$ up to a set in $\mathcal I(T)$, for our purposes can be identified with $\mathcal I(T)$.
\end{fact}
In the situation described in the above fact, let us say that \emph{$\mathcal I(T)$ almost equals $\mathcal I(T')$}.

\medskip

It is to be expected that further relationships between trees $T$ and $T'$ correspond to relationships between $\mathcal I(T)$ and $\mathcal I(T')$. For example, we can now show that a certain type of embedding between trees has consequences for the spectrum of mad families of the corresponding ideals.

We define an ordering $\prec$ on $\mathfrak T$ as follows: Let $[S] \prec [T]$ if $S'\subseteq T$ for some $S'\in [S]$.
We can now easily show the following:
\begin{lemma}\label{l.tree.embedding}
Let $T,T' \in \mathfrak T_0$ be given with $[T]\prec [T']$.
Then from any $\mathcal I(T)$-mad family $\mathcal A$ we can construct a $\mathcal I(T')$-mad family of size $\lvert\mathcal A\rvert$.
In particular, 
\[
[T]\prec [T'] \Rightarrow \mathfrak a_{\mathcal I(T')}\leq\mathfrak a_{\mathcal I(T)}
\]
\end{lemma}
\begin{proof}
Recall that $T$ and $T'$ are the underlying sets of $\mathcal I(T)$ and $\mathcal I(T')$, respectively.

We first define two maps
\begin{gather*}
A \mapsto A^*,
\powerset(T) \to \powerset(T')
\intertext{and}
A \mapsto A_*,
\powerset(T') \to \powerset(T)
\end{gather*}
as follows: 
Given $A \subseteq T$, let 
\[
A^* = \{t'\in T'\setdef (\exists t \in A)\; t\subseteq t'\}.
\]
and given $A\subseteq T'$, 
let
\[
A_*=\{t\in T\setdef (\exists t' \in A)\; t \subseteq t'\}.
\]
Note that these maps have the following properties:
\begin{enumerate}[(i)]
\item\label{i.first} $A^* \cap B^* = (A\cap B)^*$ for every $A,B \subseteq T$,
\item If $A\in \mathcal I(T)$ then $A^* \in \mathcal I(T')$,
\item If $A' \subseteq T'$ with $A' \notin \mathcal I(T')$ then $(A')_* \notin \mathcal I(T)$,
\item\label{i.last} Suppose $A \subseteq T$, and suppose $A'\subseteq T'$ satisfies 
$\{n\in\omega \setdef t\conc \la n\ra\}\notin\fin$ for each $t\in A'$. 
Then  
$(A')_* \cap B \notin \mathcal I(T) \Rightarrow A' \cap B^* \notin \mathcal I(T')$.
\end{enumerate}
Now suppose we are given an $\mathcal I(T)$-mad family $\mathcal A$. 
Let
\[
\mathcal A'=\{A^* \setdef A\in \mathcal A\}.
\]
With the above properties \ref{i.first}--\ref{i.last} it is easy to see that $\mathcal A'$ is an $\mathcal I(T')$-mad family.
For instance, to see maximality let $A'\subseteq T'$ such that $A'\notin\mathcal I(T')$ be given.
By replacing $A'$ with a subset if necessarily, we can assume that for any $t'\in A'$ it holds that
$\{n\in\omega \setdef t'\conc \la n\ra\}$ is infinite. 
Find $A \in \mathcal A$ such that $(A')_* \cap A \notin\mathcal I(T)$.
Now use \ref{i.last}.
\end{proof}

Lemma~\ref{l.tree.embedding} has the following perhaps somewhat surprising corollary. 
Given $f\colon \omega \to \omega$, let us momentarily write
$\mathcal I(f) = \bigoplus_\fin \fin^{f(n)}$. Recall that for $f,g\in{}^\omega\omega$, $f \leq^* g$ means that $\{n\in\omega\setdef f(n)>g(n)\}\in\fin$.

\newpage

\begin{corollary}
Suppose $f,g\in{}^\omega\omega$ and $f \leq^* g$. Then $\mathfrak a_{\mathcal I(g)} \leq \mathfrak a_{\mathcal I(f)}$.
\end{corollary}
\begin{proof}
This is obvious by finding trees $T_f\subseteq T_g$ such that $\mathcal I(T_f)$ almost equals $\mathcal I(f)$ and $\mathcal I(T_g)$ almost equals $\mathcal I(g)$ in the sense of Fact~\ref{f.restrict}.
\end{proof}

Our presentation of $\mathfrak F$ as $\{\mathcal I(T)\setdef T\in \mathfrak T\}$ also shows that there is potentially more than one way to extend $\la \fin^n \setdef n\in\omega\ra$ linearly into the transfinite. 
For example, fix any sequence 
\[
\vec \sigma = \la \bar \sigma^\alpha \setdef 0<\alpha < \omega_1\ra
\]
 where each %$\bar\sigma^\alpha\colon\omega \to \alpha$
 $\bar\sigma^\alpha$ is in turn a sequence from $\alpha\setminus1$ of length $\omega$,
\[
\bar \sigma^\alpha = \la\sigma^\alpha_n \setdef n\in\omega\ra,
\] 
with either supremum $\alpha$ if $\alpha$ is a limit, and so that $\bar \sigma^\alpha$ is the constant sequence with value $\alpha-1$ otherwise. % if $\alpha >0$ and the empty sequence otherwise. 
Then define 
$T^{\vec \sigma}_\alpha$ by induction on $\alpha >0$ as follows: 
\begin{gather*}
T^{\vec \sigma}_1 = {}^1\omega,\\
 T^{\vec \sigma}_\alpha = \{ \la n \ra \conc t \setdef n \in \omega, t \in T^{\vec \sigma}_{\sigma^\alpha_n}\}.
 \end{gather*}
Then $\mathcal I(T^{\vec \sigma}_\alpha)$ is exactly $\fin^\alpha$ as it was defined in \cite{karen}. 
The sequence of ideals obtained in this way extends $\la \fin^n \setdef n\in \omega\ra$. 
One should really write $\fin^{\vec \sigma,\alpha}$; 
the question what the relationship of these ideals for different choices of $\vec \sigma$ is a special case of the analogue question for $\mathcal I(T)$.

Obviously, the most all-encompassing choice for $\gamma$ would be to let $\bar\sigma^\alpha$ be an enumeration of $\alpha$ for each limit $\alpha$. 
By the fact that $T \sim T'$ implies that  $\mathcal I(T)$ and $\mathcal I(T')$ differ only trivially, in this case it is obvious that the particular choice of sequence of enumerations is of no consequence.
This is, as far as I can tell, how \cite{raghavan-steprans} defines $\fin^\alpha$.

\medskip

Let us write 
\[
\mathfrak a^{\vec \sigma}_\alpha = \mathfrak a_{\mathcal I(\fin^{\vec \sigma,\alpha})}.
\] 
and allow us to write $\mathfrak a_\alpha$ and hide the dependence on $\vec \sigma$ when it is distracting.
Lemma~\ref{l.tree.embedding} now immediately gives us:
\begin{corollary}\label{c.a}
If $\alpha < \beta < \omega_1$, $\mathfrak a_\beta \leq \mathfrak a_\alpha$, regardless of the (implicit) choice of the sequence $\vec \sigma$.
\end{corollary}
Using similar ideas, one can in fact show:
\begin{fact}
The value of $\mathfrak a_\alpha$ depends only on $\alpha$ and not on the choice of $\vec \sigma$.
\end{fact}
We are therefore justified in dropping the superscipt $\vec \sigma$ altogether and writing $\mathfrak a_\alpha$.

The following question seems to be hard:
\begin{question}\label{q.a_2}
Is it consistent with $\ZFC$ that $\mathfrak a_2 < \mathfrak a$? How about $\mathfrak a_\omega < \mathfrak a$?
\end{question}
Some light has been shed on this by Raghavan and Steprans \cite{raghavan-steprans} from which article we mention the following two results.

\begin{theorem}\label{t.steprans}
For each ideal of the form $\mathcal I = \bigoplus_\fin \mathcal I_n$, $\mathfrak b \leq \mathfrak a_{\mathcal I}$.
\end{theorem}
Recall here that $\mathfrak b$, the bounding number, is defined as
\[
\mathfrak b = \min\{ \lvert \mathcal F\rvert : \mathcal F \subseteq \omega^\omega, (\forall g \in \omega^\omega) (\exists f \in \mathcal F)\; f \mathbin{\not\leq^*} g\}.
\]
%Recall here that $\mathfrak b$ is the bounding number (see, e.g., \cite{} for a definition).
The following was shown by Raghavan and Steprans in the same article.
\begin{theorem}
For every $\alpha\in\omega_1$, $\min(\mathfrak s, \mathfrak a)\leq \mathfrak a_\alpha \leq \mathfrak a$.
\end{theorem}
Recall here that $\mathfrak s$ is the splitting number (see, e.g., \cite{handbook} for a definition).
%Recall here that $\mathfrak s$, the splitting number, is defined as
Thus, for a positive answer to Question~\ref{q.a_2} one must construct a model also of $\mathfrak b\leq\mathfrak s < \mathfrak a$.

It is well-known (and also follows from Theorem~\ref{t.steprans} and Corollary~\ref{c.a}) that $\ZF$ proves $\mathfrak b \leq \mathfrak a$. It is also known that $\mathfrak s$ is independent from each of $\mathfrak b$ and $\mathfrak a$ individually. 
Some constellations of all three are also known to be consistent (see, e.g., \cite{fischer}). The constructions tend to be difficult.
%While various results involving all three have been shown, to my knowledge there is no model of $\mathfrak b\leq\mathfrak s < \mathfrak a$. %To my knowledge, some constellations of all three are not known to be consistent.

\medskip

\subsection{The spectrum of maximal cofinitary groups and Zhang's forcing}

The cardinal invariant $\mathfrak a_{\mathrm{g}}$, i.e., the least size of a maximal cofinitary group, and more generally $\operatorname{Spec}(G_{\mathrm{mcg}})$, the possible sizes of maximal cofinitary groups, have drawn much interest. 
Very little is known about the relationship of $\mathfrak a_{\mathrm{g}}$ to other cardinal characteristics; see the introduction of \cite{good-projective} for a partial survey of results in this direction.

In connection with questions about $\operatorname{Spec}(G_{\mathrm{mcg}})$ the following forcing is very useful.
Suppose we have a cofinitary group $\mathcal G$ and we want to force to add a generic $\sigma^G \in S_\infty$ such that
the group generated by $\mathcal G\cup\{\sigma^G\}$ is also cofinitary.
If $\mathcal G$ is countable, such $\sigma^G$ is in fact added by Cohen forcing.
For any $\mathcal G$ as above, such $\sigma^G$ can be added by a forcing $\Q_{\mathcal G}$ invented by Zhang \cite{zhang-maximal};  this forcing reduces to Cohen forcing should $\mathcal G$ be countable.
Given the similarity to Cohen forcing, it is unsurprising that $\sigma^G$ is not eventually different from any permutation in the ground model with the exception of permutations in $\mathcal G$.
Thus, an iteration whose length has uncountable cofinality and which forces with $\Q_{\mathcal G}$ at a set of stages which is unbounded in the length of the iteration adds a maximal cofinitary group which extends $\mathcal G$.

\medskip

We now proceed to give an exposition of this forcing, simplifying the definition and the main arguments.

Let us fix some cofinitary group $\mathcal G$ and a letter $X$ to stand for the new generic permutation added by our forcing.
We shall need to talk about $\mathcal G * \operatorname{\mathbb{F}}(X)$, the free products of $\mathcal G$ with the free group with single generator $X$.

Clearly every element of $\mathcal G * \operatorname{\mathbb{F}}(X)$ can be written uniquely as
\begin{equation}\label{e.words}
w = g^w_{l} X^{j^w_{l-1}} \hdots g^w_0
\end{equation}
with $l=l(w) \in \omega$ and where we demand that $g^w_i \in \mathcal G\setminus\{1_{\mathcal G}\}$ for $0\leq i\leq l$ and $j^w_i\in \Int\setminus\{0\}$ for $0\leq i <l$.
We will also write $g^w_L$ for $g^w_{l(w)}$ and $g^w_R$ for $g^w_0$ (the subscripts stand for ``left-most'' and ``right-most'', of course). When $l(w)>0$, we allow either of these (but no other $g^w_i$) to equal $\emptyset$, the empty word, identified with $1_{\mathcal G}$. We also allow $l(w)=0$, in which case $w=\emptyset$, i.e., $1_{\mathcal G}$.

Writing $W_{\mathcal G,X}$ for the set of words as in \eqref{e.words} equipped with an obvious operation of con\-cat\-en\-ate-and-reduce, we obtain a useful presentation of the group $\mathcal G * \operatorname{\mathbb{F}}(X)$.

%Let us write $W^0_{\mathcal G,X}$ for the set of all words $w$ of the form
%\begin{equation*}\label{e.words}
%w = g_{l} X^{j_{l-1}} \hdots g_1 X^{j_0} g_0
%\end{equation*}
%with $l\in \omega$, $g_i \in \mathcal G$ for $0\leq i\leq l$ and $j_i\in \Int$ for $0\leq i <l$.
%(We naturally assume here that $X$ is some place-holder chosen so that $X,X^{-1}\notin \mathcal G$. Moreover, when $l=0$ we take $w$ to be the empty word, $\emptyset$)

%Clearly, with an appropriate concatenate-and-reduce operation, this becomes a group isomorphic to  
%$\mathcal G * \operatorname{\mathbb{F}}(X)$.
%Moreover, clearly, each $w\in W_{\mathcal G,X}$ can be written uniquely as
%\begin{equation*}\label{e.words}
%w = g_{l} X^{j_{l-1}} \hdots g_0
%\end{equation*}
%where we demand that $g_i \in \mathcal G\setminus\{1_{\mathcal G}\}$ for $0\leq i\leq l$ and $j_i\in \Int\setminus\{0\}$ for $0\leq i <l$.
%We will also write $g^w_L$ for $g^w_{l(w)}$ and $g^w_R$ for $g^w_0$ (the subscripts stand for ``left-most'' and ``right-most'', of course). We allow either of these (but no other $g^w_i$) to equal $\emptyset$, the empty word, identified with $1_{\mathcal G}$.

\medskip

Let us write $I_\infty$ for the set of finite injective partial functions from $\omega$ to $\omega$. 
(With the operation of concatenation of partial functions, it forms a monoid with identity.)
Let us momentarily fix $s \in I_\infty$.

Define a map $w \mapsto w[s]$ 
for $w \in W_{\mathcal G,X}$ by
\[
w[s]= g^w_{l} s^{j^w_{l-1}} \hdots g^w_0,\\
\]
where $l=l(w)$ and $s^j$ for $j\in\Int\setminus\{0\}$ denotes $s^{\sgn j}$ concatenated with itself $\lvert j\rvert$ times.
That is, $w[s]$ ``replaces each $X$ in $w$ by $s$ and each $X^{-1}$ by $s^{-1}$''.
Writing also $\rho_s$ for the map $w \mapsto w[s]$, clearly
\begin{gather*}
\rho_s\colon\mathcal G * \operatorname{\mathbb{F}}(X) \to I_\infty,\\
w \mapsto \rho_s(w)=w[s]
\end{gather*}
acts as a homomorphism of associative monoids and preserves taking inverses; in fact, it is the unique homomorphism of groupoids which restricts to the identity on $\mathcal G$ and sends $X$ to $s$.

Before we can define $\Q_{\mathcal G}$ we need one last piece of terminology:
Let us say that $w_0 \in W_{\mathcal G, X}$ is a proper conjugated subword of $w_1 \in W_{\mathcal G, X}$ if there exists $w \in W_{\mathcal G, X}\setminus\{\emptyset\}$ such that $w_1=w^{-1}w_0 w$.

\begin{definition}[The forcing $\Q_{\mathcal G}$]\label{d.Q.simple}~ 
\begin{enumerate}[label=(\alph*),ref=\alph*]
\item
Conditions of $\Q_{\mathcal G}$ are pairs $p=(s^p,F^p)$ where $s \in I_\infty$ and $F^p \subseteq W_{\mathcal G, X}$ is finite %, closed under taking subwords, 
and contains only words without proper conjugated subwords.
\item
$(s^q,F^q)\leq_{\Q_{\mathcal G}} (s^p,F^p)$ if and only if $s^q \supseteq s^p$, $F^q \supseteq F^p$ and for all $w\in F^p\setminus\mathcal G$,
$\fix(w[s^q]) = \fix(w'[s^p])$. \label{d.Q.simple.order}
\end{enumerate}
\end{definition}
Transitivity is sometimes tedious to check; but not here.
\begin{lemma}
The ordering $\leq_{\Q_{\mathcal G}}$ is transitive.
\end{lemma}
\begin{proof}
With the present definition of $\leq_{\Q_{\mathcal G}}$ this is completely obvious.
\end{proof}

We write any condition  $p \in \Q_{\mathcal G}$ as $(s^p, F^p)$ if we want to refer to the components of that condition.

If $G$ is $(\Ve,\Q_{\mathcal G})$-generic, we let
\[
\sigma^G = \bigcup_{p\in G} s^p.
\]
Using requirement \eqref{d.Q.simple.order} and typical genericity arguments it is easy to see that 
\[
\rho_{\sigma^G}\colon\mathcal G*\bF(X)\to \langle \mathcal G, \sigma^G \rangle
\] 
is an isomorphism of groups, and the image on the right is a cofinitary group. 
Moreover, $\langle \mathcal G, \sigma^G \rangle$ is maximal with respect to the ground model, i.e., for no $\tau\in (S_\infty\setminus\mathcal G) \cap \Ve$ is $\mathcal G\cup\{\sigma^G, \tau\}$ contained in a cofinitary group.

\medskip

Before we continue the discussion of $\Q_{\mathcal G}$, it is convenient to introduce the notion of \emph{path}, which despite being fairly intuitive should be given a precise definition. 
\begin{definition}[Paths]
Given $s\in I_\infty$ and a word $a_l\hdots a_1\in W_{\mathcal G,X}$ where each $a_i \in \mathcal G\cup\{X, X^{-1}\}$, \emph{the path of $m\in\omega$ under $(w,s)$} is the sequence $\la m_k:k\in\alpha\ra$, where $m_0=m$, and where for each $k$, writing $k=nl+i$ with $i<n$ we have
$$m_k=(a_i\cdots a_1w^{nl})[s](m),$$
and where $\alpha$ is either $\omega$, or denotes the least $k$ for which $m_{k}$ as above is not defined.
In other words, the path is the sequence given by the following evaluations:
\[
\hdots m_{k+1}\xleftarrow{\ a_i} m_{k}\hdots\stackrel{\ a_2}{\longleftarrow} m_{l+1} \stackrel{\ a_1}{\longleftarrow} m_l \stackrel{\ a_l}{\longleftarrow}  \hdots  \stackrel{\ a_2}{\longleftarrow} m_1 \stackrel{\ a_1}{\longleftarrow} m_0=m
\]
\end{definition}

\medskip

We return to our  discussion of $\Q_{\mathcal G}$.
First, let us explain why we do not allow arbitrary words with proper conjugated subwords in $F^p$:
For supposing, e.g., $g\in\mathcal G$ and $n\in\fix(g)$ the condition $(\emptyset, \{ X^{-1}gX\}) \in\Q_{\mathcal G}$ has no extension $q\in\Q_{\mathcal G}$ with $n\in \ran (s^q )$. 
More generally, if $s\in I_\infty$ and $w_0,w_1,w' \in W_{\mathcal G, X}$ are such that $n\in\fix(w'[s])$ but $w_0[s](n)\notin\dom(s)$,  
then $(s, \{ w_1Xw_0w'(w_1Xw_0)^{-1}\}) \in\Q_{\mathcal G}$ has no extension $q\in\Q_{\mathcal G}$ with $w_0[s](n)\in \dom (s^q )$.
A similar example can be found with $X^{-1}$ and $\ran(s)$. 
This obstruction was already pointed out in \cite[p.~42f.]{zhang-maximal}.

\medskip

Previously, the strategy to avoid this obstruction has been to adopt a more complicated definition of $\leq_{\Q_{\mathcal G}}$, namely as follows.

\begin{enumerate}[label={$(\ref{d.Q.simple.order}')$},ref={$\ref{d.Q.simple.order}'$}]
\item\label{d.Q.simple.order'}
$(s^q,F^q)\leq_{\Q_{\mathcal G}} (s^p,F^p)$ if and only if $s^q \supseteq s^p$, $F^q \supseteq F^p$ and for all  $w\in F^p$ and $m\in\hbox{fix}(w[s^q])$, %then there is a non-empty subword $w^\prime$ of $w$ such that \todo{Define $\operatorname{fix}(x)$}
%$$\operatorname{use}(w,t,m)\cap\hbox{fix}(w^\prime[s])\neq\emptyset;$$
%and for all $w\in F^p\setminus\mathcal G$,
%if $m \in \fix(w[s^q])$ 
there is a non-empty subword $w'$ of $w$ such that
letting $w=w_1w' w_0$ and letting $\langle \hdots m_1, m_0 \rangle$ be the $(w,s^q)$-path of $m$,
$m_k \in \fix(w'[s^p])$ where $k$ is the length of $w_0$. We describe this situation by the following picture:
\[
m \xleftarrow{\ w_1} m_{k} \xleftarrow{\ w'} m_{k}\stackrel{\ w_0}{\longleftarrow} m
\]
\end{enumerate}

With the above variant definition of $\leq_{\Q_{\mathcal G}}$ it becomes possible to allow arbitrary words in $F^p$, including words with proper conjugated subwords. But since conjugate pairs of words have the same fixed points,
it is enough to rule out infinitely many fixed points for words without proper conjugate subwords.

\medskip

The following is a much more streamlined version of the crucial lemma in \cite{zhang-maximal} (a version of this exists in \cite{good-projective}).

\medskip

To state the lemma, let us clarify what we mean by the \emph{circular shift} of a word in $w \in W_{\mathcal G,X}$:
Writing $w=w_l\cdots w_1$ in the form as in \eqref{e.words} and given a permutation $\sigma:\{1,\cdots, l\}\to\{1,\cdots,l\}$ such that
$\sigma(i)=i + k\hbox{ mod }l$ for some $k\in \omega$, we will refer to $w_{\sigma(l)}\cdots w_{\sigma(1)}$ as a circular shift of $w$. Thus, in particular, there are only finitely many circular shifts of a given word.

\begin{lemma}[Domain Extension for $\Q_{\mathcal G}$]\label{l.domain.ext}
Suppose $s\in I_\infty$, $w\in W_{\mathcal G, X}$ has no proper conjugate subwords and $n\in\omega\setminus\dom(s)$. 
Then for a co-finite set of $n'$, letting $s' = s \cup \{ (n,n')\}$, we have that $s'$ is injective and $\fix(w[s'])=\fix(w[s])$.
\end{lemma}
\begin{proof}
Let $W^*$ be the set of subwords of circular shifts of $w$ and pick $n'$ arbitrary such that
\begin{equation}
\begin{split}\label{e.domain.ext.cont}
n' \notin  &\bigcup \big\{ \fix(w'[s]) \colon w' \in W^*\setminus\{ \emptyset\}\big\},\\
n' \notin &\bigcup  \big\{ w'[s]^i (n) \colon i \in \{-1,1\}, w' \in W^* \big\},\text{ and}\\
n' \notin &\ran(s).
\end{split}
\end{equation}
By the last requirement $s'$ is injective and by the middle requirement $n' \neq n$ since $\emptyset \in W^*$. 

Assume towards a contradiction that $m_0 \in \fix (w[s']) \setminus \fix(w[s])$.
As the $(w,s)$-path of $m_0$ differs from the $(w,s')$-path, the latter must contain an application of $X$ to $n$ or of $X^{-1}$ to $n'$. Write this latter path (omitting some steps) as
\begin{equation}\label{e.path.basic}
m_0 \stackrel{w_{l+1}}{\longleftarrow} m_{k(l)+1} \stackrel{X^{j(l)}}{\longleftarrow} m_{k(l)} 
\stackrel{w_{l}}{\longleftarrow} \hdots
\stackrel{w_1}{\longleftarrow} m_{k(0)+1} \stackrel{X^{j(0)}}{\longleftarrow} m_{k(0)} 
\stackrel{w_{0}}{\longleftarrow} m_0
\end{equation}
where for each $i \leq l$, $j(i) \in \{ -1,1\}$ and $\{ k(i) \colon i \leq l\}$ is the increasing enumeration of the set of $k$ such that
$m_k =n$ and $X$ is applied or $m_k = n'$ and $X^{-1}$ is applied at step $k$.
Thus by definition $w_i[s] = w_i [s']$ for each $i \leq l +1$.
%all the applications of $X$ to $n$ or $X^{-1}$ to $n'$ are displayed.
%That is, 
%Thus, from $m_{k(i)}$ to $m_{k(i+1)}$, the path contains no application of $X$ to $n$ or of $X^{-1}$ to $n'$,
%and each $w_i$ is a maximal subword of $w$ with that property.
%Moreover apart from $w_i$ possibly being empty there is no cancellation in \eqref{e.path}.

The following hold by definition of the $k(i)$ and by choice of $n'$:
\begin{enumerate}[(i)]
%\item As the $(w,s)$-path of $m_0$ differs from the $(w,s')$-path, the latter must contain an application of $X$ to $n$ or of $X^{-1}$ to $n'$; thus $n > 1$.
\item Unless $i=0$ or $i=l+1$, $w_i \neq \emptyset$: for assuming otherwise,  $j(i-1) = j(i)$ is impossibly as $n \neq n'$ ; but $j(i-1) \neq j(i)$ is also impossibly as adjacent $X$ and $X^{-1}$ cannot cancel (we assume $i\notin\{0,l+1\}$ here because of course, no such cancellation occurs at the word boundary).
\item\label{i.2} For no $i \leq l$ is it the case that $w_{i}$ sends $n$ to $n'$ or vice versa.
This is by choice of $n'$ as in \eqref{e.path.basic}.
\item\label{i.3} For no $i \leq l$ is $n'$ a fixed point of $w_i[s]$ unless $w_i = \emptyset$, again by choice of $n'$ as in \eqref{e.path.basic}.
\end{enumerate}
From this it follows that if $0 < i < l+1$, the values in the path appearing adjacent to $w_i$ (i.e., $m_{k(i-1)+1}$ and $m_k(i)$) are both $n$. 
%From this we easily conclude the following
%\begin{enumerate}
%\item It is impossible that $n=1$:
%for then $w_0 w_n$---a subword of a cyclic shift of $w$---sends $n$ to $n'$ or $n'$ to $n$ both of which is impossible by choice of $n'$.
%\item Thus $j(2)$ is defined; as $n'$ can appear on neither side of $w_2$ we have $n \in \fix(w_2)$,
%$m_{k(1)} = m_{k(2)} = n$, $j(0) = -1$, and $j(1)=1$.
%\item $j(4)$ cannot be defined as otherwise, $w_3$ must be non-empty and send $n'$ to $n'$.
%\end{enumerate}
There is at least one such $i$, for the path cannot have the following form:
\begin{equation}\label{e.path.not}
m_0 
\stackrel{w_1}{\longleftarrow} m_{k(0)+1} \stackrel{X^{j(0)}}{\longleftarrow} m_{k(0)} 
\stackrel{w_{0}}{\longleftarrow} m_0
\end{equation}
for then $w_0 w_1$---a subword of a cyclic shift of $w$--- or its inverse sends $n$ to $n'$ which is impossible by choice of $n'$.
Thus $w_2$ and $j(1)$ are defined; as $n'$ can appear on neither side of $w_1$ we have $n \in \fix(w_1)$,
$m_{k(1)} = m_{k(1)+1} = n$, $j(0) = -1$, and $j(1)=1$. (Observe that here the proof is done if we work with the definition of our partial order as given in \eqref{d.Q.simple.order'}).

Finally $j(2)$ cannot be defined as otherwise $w_3$ must be non-empty and send $n'$ to one of $\{ n, n'\}$, contradicting Items \ref{i.2} or \ref{i.3} above.
So the path in \eqref{e.path.basic} has the following form:
\begin{equation}\label{e.path.semifinal}
m_0 
\stackrel{w_3}{\longleftarrow} n' \stackrel{X^{-1}}{\longleftarrow} n 
\stackrel{w_2}{\longleftarrow} n \stackrel{X}{\longleftarrow} n' 
\stackrel{w_{1}}{\longleftarrow} m_0
\end{equation}
As $w_3 w_1$ is a subword of a cyclic shift of $w$,
$w_3 w_1 = \emptyset$ since we made sure $n' \notin \fix(w_3 w_1[s])$ otherwise.
So $w_3 = {w_1}^{-1}$ and $w_2$ is a conjugate subword of $w$.
In fact, by the presence of $X$ and $X^{-1}$ adjacent to $w_2$, it is a proper conjugate subword.
Since we assumed that $w$ had no proper conjugate subword, we reach a contradiction.
\end{proof}

\section{Optimal projective witnesses}

Questions about the spectrum of a hypergraph $G$ other than to give a particular value to its minimum are of interest.
For example, one can ask whether it is consistent with $\neg\CH$ that $\operatorname{Spec}(G) = \{\mathfrak a_G, 2^\omega\}$, or how to find maximal discrete set whose size lies strictly between the minimum and $2^\omega$ in $\operatorname{Spec}(G)$.

In particular, one might be interested in the minimum definitional complexity of a maximal discrete set $D$ of a given size $\kappa$ from the spectrum of $G$.
\begin{definition}
Fix a hypergraph $G=(X,H)$ and a cardinal $\gamma \in \operatorname{Spec}(G)$.
We say $D$ is a \emph{$\Pi^1_n$ (resp.\ $\Delta^1_n$) witness to $\gamma$} if $D$ is $\Pi^1_n$ (resp.\ $\Delta^1_n$) maximal $G$-discrete set of size $\gamma$.
We say such a witness is \emph{optimal} if there is no $\Sigma^1_{n-1}$ or $\Pi^1_{n-1}$ maximal $G$-discrete $D'$ set  with $\lvert D\rvert = \kappa$---i.e., there is no witness to $\gamma$ of strictly lower complexity in terms of the projective hierarchy.
The term \emph{(optimal) projective witness (to $\gamma$)} has the obvious meaning, i.e., a $\Sigma^1_{n-1}$ or $\Pi^1_{n-1}$ such witness for some $n\in\omega$.
\end{definition}

\begin{conjecture}
Suppose $\GCH$ and that $\omega_1 < \gamma<\kappa$ are cardinals of uncountable cofinality. Then there is a cardinal preserving forcing extension of $\Ve$ in which 
$2^\omega=\kappa$ and there is a projective witness to $\gamma$. 
In fact there exists a good $\Pi^1_2$ witness to $\gamma$.
\end{conjecture}
A version of this conjecture has been shown in \cite{good-projective}; unfortunately, this version restricts $\gamma, \kappa$ to lie below $\aleph_\omega$. One reason for this is that the proof of the full conjecture would presumably use Jensen coding, which is notoriously cumbersome to work with.

\medskip

The following loosely related question is open:
Is it consistent with $\ZF + \DC + \neg\CH$ that every cofinitary group $\mathcal G$ with $\lvert \mathcal G\rvert < 2^\omega$ can be extended to a projective maximal cofinitary group?

This question we can answer as follows:
\begin{theorem}
It is consistent with $\ZF + \DC + \neg\CH$ that every cofinitary group $\mathcal G$  with $\lvert \mathcal G\rvert < 2^\omega$ can be extended to a $\mathbf{\Pi}^1_2$ maximal cofinitary group.
\end{theorem}
In the model witnessing the above, there is a $\mathbf{\Pi}^1_2$ witness to every cardinal of uncountable cofinality in $\operatorname{Spec}(G_{\mathrm{mcg}})$.

\medskip

To prove this and similar theorems, or generally when one wants to create optimal projective witnesses for $G_{\mathrm{mcg}}$, a version of Zhang's forcing with ``built-in coding'' becomes important. We shall now present such a forcing (for a comprehensive treatment, see \cite{good-projective}; more or less distant ancestors can also be found in \cite{mcg-cohen,kastermans}).

\subsection{Zhang's forcing with coding}

Fix a cofinitary group $\mathcal G$.
Our goal is again to add a generic permutation $\sigma$ such that $\la \mathcal G\cup \{\sigma\}\ra$ is cofinitary and maximal with respect to permutations from the ground model.
This time, we want to make sure that each permutation $g\in \la \mathcal G\cup \{\sigma\}\ra\setminus \mathcal G$ 
``codes'' a real $z^g$ in the sense that $z^g$ is computable from $g$---similar to the assumption made in Theorem~\ref{t.zoltan}.
The map $g \mapsto z^g$ will be given at the outset, in the ground model, in the form of a map 
\[
\bar z\colon\mathcal G * \mathbb{F}(X) \setminus \mathcal G \to {}^\omega 2.
\]
%where we also write $z^w$ for $\bar z(w)$. 

\medskip

Before we can define the forcing, we need to fix the algorithm used in the coding of $z^{w}$ by the generic permutation $w[\sigma^G]$.
\begin{definition}[Coding]\label{d.coding} 
%From now on, let $\langle p_n : n\in\omega\rangle$ denote some recursive partition of $\nat$.
%For $k\in\omega$ and $w \in \mathrm{W}^{a}$, define 
%\[
%S(k,w)=\begin{cases}
% 0 &\text{ if $k=0$,}\\
% p_{\lh(w)}^{k} &\text{ if $k>0$.}
%\end{cases}
%\]
Let a sequence $\chi \in 2^{\leq\omega}$ be given. 
Suppose $\sigma$ is a partial function from $\omega$ to $\omega$, $w\in W_{\mathcal G,X}$ has no proper subwords $w'$ such that $w=(w')^n$ as well as no proper conjugate subwords, and $X$ or $X^{-1}$ occurs at least once in $w$.
\begin{enumerate}

\item
We say $(w,\sigma)$ \emph{codes} $\chi$ \emph{with parameter} $m$ if and only if 
\begin{equation}\label{e.code}
(\forall k < \lh(\chi)) \; w^{3k}[\sigma](m) \equiv \chi(k) \pmod{2}.
\end{equation}
%\begin{equation}\label{e.code}
%(\forall k < \lh(\chi)) \; w^{p_{\lh(w)}^{k+1}}[\sigma](m) \equiv \chi(k) \pmod{2}.
%\end{equation}
\item Suppose now that $\lh(\chi) <\omega$. 
%Write $w=w_1 w_0$ where $w_0$ is shortest so that its leftmost letter is $a$ or $a^{-1}$.
We say that  $(w,\sigma)$  \emph{exactly codes} $\chi$
\emph{with parameter} $m$ if $(w,\sigma)$ codes $\chi$ 
and in addition
%\[                  %THIS DOESN'T WORK BECAUSE p^0=1
%a^i g^w_R w^{p_{\lh(w)}^{\lh(\chi)}}[\sigma](m)\text{  is undefined,}
%\]    
\[
X^i g^w_R w^{3k}[\sigma](m)\text{  is undefined,}
\]
where $i$ is the sign of $j^w_{0}$, i.e., the exponent of the right-most occurrence of $X$ or $X^{-1}$ in $w$.
In other words, at least for $\lh(\chi)>0$, $(w,\sigma)$ exactly codes $\chi$ if the path of $m$ under $(w,\sigma)$ is of minimal possible length under the requirement that it codes $\chi$.
We require this so we have enough freedom to separate different coding paths for incomparable words in the density argument showing that it is forced that $w[\sigma_{\dot G}]$ codes $\bar z(w)$.
%\item 
%We say that  $(w,\sigma)$  \emph{exactly codes} $\emptyset$
%\emph{with parameter} $m$   
%\[
%a^i g^w_R [\sigma](m)\text{  is undefined,}
%\]
%where $i$ is defined as in the previous item.
%This clause is made necessary because the previous clause does not make sense if $\lh(\chi)=0$, but we need an exactness requirement to start the coding when a condition does not yet code a non-trivial string.
\item We say that $m'$ is \emph{the critical point in the path of $m$ under $(w,\sigma)$} if for some $k\in\omega$,
\[
m' = (g^w_L X^i)^{-1}w^{3(k+1)}[\sigma](m) 
\]
%\[
%m' = (g^w_L a^i)^{-1}w w^{p_{\lh(w)}^{k+1}}[\sigma](m) 
%\]
where $i$ is the sign of $j^w_{l(w)-1}$, i.e., the exponent of the left-most occurrence of $X$ or $X^{-1}$ in $w$.
We define this terminology because, when extending $\sigma$ so that the path of $m$ increases in length with the purpose of achieving exact coding of a given $\chi$, it is precisely at critical points that exact coding imposes a non-trivial requirement for this extension.
\end{enumerate}
\end{definition}

Let $\mathcal G'$ denote the set of elements of $w\in\mathcal G * \mathbb{F}(X)$ which have no proper conjugate subword, so that for no proper subword $w'$ we have $w=(w')^n$, and so that  
$X$ or $X^{-1}$ occurs at least once in $w$. 
Now in addition to the cofinitary group $\mathcal G$, suppose we are given
a map 
\[
\bar z\colon \mathcal G' \to {}^\omega 2.
\]
%and write $z^w$ for $\bar z(w)$. 

\begin{definition} Let the forcing $\Q^{\bar z}_{\mathcal G}$ consist of conditions $p=(s^p,F^p, \bar m^p)$ where 
$(s^p,F^p) \in\Q_{\mathcal G}$ and $\bar m^p$ is a finite partial function from $\mathcal G'$ to $\omega$, and moreover $p$ satisfies
\begin{enumerate}[label=(\Alph*),ref=\Alph*,start=5]
\item\label{Q.exact} for each $w \in \dom(\bar m^p)$ there exists a (unique) $l$ which we denote by $l^p_w$ such that $(w,s^p)$ exactly codes $\bar z(w)\res l$ with parameter $\bar m^p(w)$.
\end{enumerate}

The ordering on $\Q^{\bar z}_{\mathcal G}$ is as follows: 
Let $q \leq_{\Q^{\bar z}_{\mathcal G}} p$ if $(s^q,F^q) \leq_{\Q_{\mathcal G}} (s^p,F^p)$ in $\Q_{\mathcal G}$ and $\bar m^p \subseteq \bar m^q$.
\end{definition}

One can then show the following:

\begin{theorem}\label{prop.group.forcing}
Suppose that $\mathcal G$ is a cofinitary group with the following property:
For each $n\in\omega$, $i_0, \hdots, i_n \in \{0,1\}$ and $g_0, \hdots, g_n \in \mathcal G$, there are infinitely many $m$ such that  for each $j\leq n$ it holds that 
$g_j(m) \equiv i_j \pmod 2$.
Let $G$ be a $\Q^{\bar z}_{\mathcal G}$-generic filter and let
$$\sigma^G=\bigcup\{s^p \setdef p\in G \}.$$
The permutation $\sigma^G$ has the following properties:

\begin{enumerate}[label=(\Roman*)]
\item\label{i.A} The group $\la \mathcal G\cup\{\sigma^G \}\ra$ is cofinitary.

\item\label{i.B} If $\tau$ is a ground model permutation, $\tau\notin \mathcal G$,  and $\la \{\tau\}\cup{\mathcal G}\ra$ is cofinitary, there are infinitely many $n$ such that
$\tau(n)=\sigma^G(n)$ and so $\la{\mathcal G}\cup\{\sigma^G\}\cup\{\tau\}\ra$ is not cofinitary;

\item\label{i.C} For each $w\in \mathcal G'$ there is $m\in \omega$ such that
$w[\sigma^G]$ codes $\bar z(w)$ with parameter $m$.
%for all $k\in\omega$, $w^{2k}[\sigma^G ](m_w)=\chi_{\bar z(w)}(k)\mod 2$ (where of course $\chi_z$ denotes the characteristic function of $z$). 

\end{enumerate}
\end{theorem}
We omit all proofs regarding this forcing; but we shall try to give at least an idea of how the forcing generically achieves the coding of $\bar z(w)$ by $w[\sigma^G]$.

\begin{conjecture}[Generic Coding]\label{lemma.generic.coding}
Suppose $w\in W_{\mathcal G,X}$ with at least one occurrence of $X$ or $X^{-1}$, without proper conjugate subwords, and so that for no proper subword $w'$ of $w$ does it holds that $w=(w')^n$ (for some $n$).
For $l \in \omega$, let $D^{\textup{code}}_{w,l}$ denote the set of $q\in\Q^{\bar z}_{\mathcal G}$ such that $w \in \dom(\bar m^q)$ and for some $l' \geq l$,
$q$ exactly codes $\bar z(w)\restriction l'$ with parameter $\bar m^q(w)$. Then $D^{\textup{code}}_{w,l}$ is dense in $\Q^{\bar z}_{\mathcal G}$.
\end{conjecture}
\begin{proof}[Sketch of proof idea.]
Let $w$ and a condition $q \in\Q^{\bar z}_{\mathcal G}$ be given and write $l=l^q_w$; it suffices to find a stronger condition $r$ such 
that  $l^r_w = l +1$.

To achieve $l^r_w = l +1$ we must ensure exact coding of $\bar z(w) \res l+1$.
The only difficulty is that there may be other words $w' \in \dom(\bar m^q)$ whose coding paths merge with the coding path of $w$; here, by a coding path of a word $w'$ we mean the path under $(w',s^q)$ of $\bar m^q(w')$.

So our task is to extend $s^q$ finitely many times to obtain $s^r$ in such a way that with each extension, we separate the coding paths of each of the relevant words $w' \in \dom(\bar m^q)$, if possible. 
To do this one uses the same method as in the proof of the Domain Extension Lemma~\ref{l.domain.ext}.

It turns out that the only coding paths which cannot be made to diverge before reaching a critical point are those of words $w', w'' \dom(\bar m^q)$ such that $w'$ is a subword of $w''$ or vice versa.
This is because if $w'$ is $\subseteq$-incomparable to $w''$ then before reaching the next critical point in the coding path of $w'$ we must reach a point where $w'$ and $w''$ differ in the next letter to be applied. 
So we take successive extensions of $s^q$ in such a way that the coding paths associated to $\subseteq$-incomparable $w',w''\in\dom(\bar m^q)$ separate when they reach the first letter where $w'$ itself diverges from $w''$  (here one uses that $\mathcal G$ is cofinitary).

It remains to see that none of the exact coding requirements for the $\subseteq$-comparable pairs of words $w', w''\in \dom(\bar m^q)$ create a conflict. 
For this, it might be necessary to make additional assumptions on $\mathcal G$; in any case, the full proof will be combinatorially involved.
In particular, one shows that the only coding paths which cannot be made to diverge before reaching a critical point are those of words $w', w'' \dom(\bar m^q)$ which differ only in their first and last letters from $\mathcal G$. 
Since any two such words differ in their last letter, we can use the property of $\mathcal G$ in the hypothesis of Theorem~\ref{prop.group.forcing} to simultaneously satisfy all coding requirements.
See the upcoming \cite{good-projective} for details.
\end{proof}
\medskip

\noindent
{\it 
The present version of this article contains some corrections in comparison to the published version.
}

\medskip
\noindent
{\it Acknowledgments:
The author was supported by Vera Fischer's FWF Start Grant Y1012.
}

\bibliography{monsters}{}
\bibliographystyle{amsplain}

 \end{document}